\documentclass[10pt]{amsart}
\usepackage{amsmath,amssymb,amsthm}
\usepackage{graphicx,url}
\newtheorem{theorem}{Theorem}    
\newtheorem{lemma}{Lemma} 
\newtheorem{claim}{Claim} 
\newtheorem{proposition}{Proposition} 
 
\newtheorem{question}{Question} 

\theoremstyle{definition}

\newtheorem*{remark*}{Remark}
\newcommand{\Z}{\mathbb{Z}}

 \makeatletter
    
    \@addtoreset{equation}{section}
  \makeatother

\title[Chirally cosmetic surgery on cable knots]{A note on chirally cosmetic surgery on cable knots}
\author[T.Ito]{Tetsuya Ito}
\address{Department of Mathematics, Kyoto University, Kyoto 606-8502, JAPAN}
\email{tetitoh@math.kyoto-u.ac.jp}
\subjclass[2010]{Primary~57M25, Secondary~57M27}
\keywords{}

\begin{document}

\begin{abstract}
We show that a $(p,q)$-cable of a non-trivial knot $K$ does not admit chirally cosmetic surgery for $q\neq 2$, or $q=2$ with additional assumptions. In particular, we show that $(p,q)$-cable of non-trivial knot $K$ does not admit chirally cosmetic surgery as long as the JSJ piece of knot exterior does not contain $(2,r)$-torus exterior. We also show that an iterated torus knot other than $(2,p)$-torus knot does not admit chirally cosmetic surgery.
\end{abstract}

\maketitle

\section{Introduction}
Let $S^{3}_K(m \slash n)$ be the Dehn surgery along a knot $K$ in $S^{3}$ of slope $m\slash n$. Two Dehn surgeries $S^{3}_K(m\slash n)$ and $S^{3}_K(m'\slash n')$ are \emph{purely cosmetic} (resp. \emph{chirally cosmetic}) if $S^{3}_K(m\slash n) \cong S^{3}_K(m'\slash n')$ (resp. $S^{3}_K(m\slash n) \cong -S^{3}_K(m'\slash n')$). Here for an oriented 3-manifold $M$ we denote by $-M$ the same 3-manifold with opposite orientation, and $M \cong N$ means that they are orientation preservingly homeomorphic.

It is expected that a non-trivial knot $K$ in $S^{3}$ does not have purely cosmetic surgeries (cosmetic surgery conjecture \cite[Problem 1.81 (A)]{Kirby}), whereas there are two families of chirally cosmetic surgeries on non trivial knots;
\begin{itemize}
\item[(a)] For amphicheiral knot $K$, $S^{3}_K(m\slash n) \cong -S^{3}(-m\slash n)$.
\item[(b)] for $(2,r)$-torus knot $K$ we have
$S^{3}_K(\frac{2r^{2}(2m+1)}{r(2m+1)+1})\cong -S^{3}_K(\frac{2r^{2}(2m+1)}{r(2m+1)-1})$ for any $m\in \Z$.
\end{itemize}

Since currently no other examples of chirally cosmetic surgery of knots in $S^{3}$ are known, one encounters a natural question.
\begin{question}
\label{question:main}
Is chirally cosmetic surgery of knots in $S^{3}$ either (a) or (b) ?
\end{question}

At first glance this may sound too optimistic since there are several unexpected phenomenon or clever constructions which negate naive conjectures on Dehn surgeries. Moreover, when we extend our attention to knots in general 3-manifolds $M$, there are more examples of chirally cosmetic surgeries which are not generalizations of above examples (a) and (b) \cite{BHW,IJ}. 

Nevertheless, recently we observed some results supporting the affirmative answer to this question \cite{Ito,IIS}. In this note we show a non-existence of chirally cosmetic surgery for cable knots under some technical assumptions. 

Let $E(K)= S^{3} \setminus N(K)$ be the knot exterior, where $N(K)$ denotes an open tubular neighborhood of $K$. There is a family of essential tori $\mathcal{T}=\{T_1,\ldots,T_n\}$ (possibly empty) of $E(K)$ such that each component of $E(K)\setminus \mathcal{T} := E(K) \setminus (\bigcup_i T_i)$ is geometric (i.e., either hyperbolic or Seifert fibered). Such a family of tori $\mathcal{T}$, called the \emph{JSJ tori}, is unique up to isotopy when we take a minumum one. We call a connected component $X$ of $E(K) \setminus \mathcal{T}$ the \emph{JSJ piece} of $E(K)$.

\begin{theorem}
\label{theorem:main}
Let $K_{p,q}$ be the $(p,q)$-cable of a non-trivial knot $K$. 
Assume that one of the following conditions are satisfied.
\begin{itemize}
\item[(i)] $q\neq 2$.
\item[(ii)] $q=2$, $p\neq \pm 1$, and the JSJ piece of $E(K)$ does not contain $(-p,2)$-torus knot exterior.
\item[(iii)] $q=2$, $p=\pm 1$ and the JSJ piece of $E(K)$ does not contain $(r,2)$-torus knot exterior for any $r$.
\item[(iv)] $q=2$, $p=\pm 1$ and $a_{2}(K) \neq 0$.
\end{itemize} 
Then $K_{p,q}$ does not admit chirally cosmetic surgeries.
\end{theorem}

Here $a_2(K)$ is the coefficient of $z^{2}$ for the Conway polynomial $\nabla_K(z)$. We remark that in our notation the $(p,q)$ cable $K_{p,q}$ of $K$ is defined so that it has wrapping number $q$; $K_{p,q}$ intersects with $\{pt\} \times D^{2} \subset S^{1} \times D^{2} \cong N(K)$ at $q$ points.

We mention that a non-existence of purely cosmetic surgery of cable knots are shown in \cite{Tao}. Although there are many similarities we do not use this result. Indeed, a mild modification of the proof of Theorem \ref{theorem:main} proves a non-existence of purely cosmetic surgery on cable knots.

In a light of an example (b) of chirally cosmetic surgery and Theorem \ref{theorem:main}, one may think that an iterated torus knot is a possible candidate for new chirally cosmetic surgeries. However, we show that iterated torus knots does not admit chirally cosmetic surgery.

\begin{theorem}
\label{theorem:main2}
An iterated torus knot which is not a $(2,p)$-torus knot does not admit chirally cosmetic surgeries.
\end{theorem}

\section*{Acknowledgement}
The author has been partially supported by JSPS KAKENHI Grant Number 19K03490,16H02145. He is grateful to Kazuhiro Ichihara for invaluable comments and
discussions.

\section{Dehn surgery of cable knots}
\label{section:Dehn}

For a torus boundary component $T$ of a 3-manifold $X$, a \emph{slope} $\gamma$ (on $T$) is an isotopy class of a non-trivial unoriented simple closed curve on $T$. We take an ordered basis $(\alpha,\beta)$ of $H_{1}(T;\Z)$ to identify the set of slopes with $\mathbb{Q}\cup \{\infty=\frac{1}{0}\}$; We view $\gamma$ as an oriented simple closed curve by taking one of its orientation, its homology class is written by $[\gamma]=p\alpha +q\beta \in H_{1}(T;\Z)$ for coprime integers $p$ and $q$. Then we assign the slope $\gamma$ a rational number $p \slash q \in \mathbb{Q} \cup \{\infty=\frac{1}{0}\}$ (note that $p$ and $q$ depend on a choice of orientation, whereas $p\slash q$ does not).

In a case of knot complement $E(K)$, we take the standard meridian-longitude pair $([\mu],[\lambda])$ as an ordered basis of $H_{1}(\partial E(K);\Z)$. The $m\slash n$-surgery on $K$ is the 3-manifold $S^{3}_K(m\slash n)$ obtained from $E(K)$ by attaching the solid torus $S^{1}\times D^{2}$ along $\partial E(K)$ so that the slope $m\slash n$ bounds a disk in the attached solid torus.

The $(p,q)$ torus knot $T_{p,q}$ is a slope $\frac{p}{q}$ curve on a boundary of the standardly embedded solid torus $S^{1}\times D^{2}$ in $S^{3}$, with respect to the basis $([\{\ast\}\times \partial D^{2}], [S^{1} \times \{\ast\}])$ of $H_{1}(\partial S^{1} \times D^{2};\mathbb{Z})$.
Thus in our convention the $(p,q)$-torus knot $T_{p,q}$ is the closure of the $q$-braid $(\sigma_1\cdots \sigma_{q-1})^{p}$. 
In the following, we will often view $T_{p,q}$ as a knot in the solid torus $S^{1} \times D^{2}$.

The $(p,q)$-cable $K_{p,q}$ of the knot $K$ is the image $f(T_{p,q})$ of the standard torus knot $T_{p,q} \subset S^{1} \times D^{2}$, where $f: S^{1} \times D^{2} \rightarrow \overline{N(K)}$ is a homeomophism such that $f(S^{1}\times \{*\}) = \lambda, f(\{*\} \times \partial D^{2})=\mu$, and $\overline{N(K)}$ denotes the closure of $N(K)$.
Since $K_{p,q} =K$ if $q=1$, in the following we always assume that $q>1$.

By \cite{Gordon}, the Dehn surgery along cable knot is described as follows;
\[ S^{3}_{K_{p,q}}(m\slash n) = \begin{cases} S^{3}_{K}(p\slash q) \# L(q,p) & |npq-m|=0, \\
S^{3}_K(m\slash nq^{2}) & |npq-m|=1, \\
E(K)\cup_{T} P_{p,q,m,n} & |npq-m|>1.
\end{cases}
\]
In the last case $P_{p,q,m,n}$ is a Seifert fibered space with base surface $D^{2}$ having two singular fibers, glued along the boundary $T:=\partial E(K)$ of $E(K)$. Moreover $P_{p,q,m,n}$ is a JSJ piece of $S^{3}_{K_{p,q}}(m\slash n)$.

In the following we prove Theorem \ref{theorem:main} by dividing arguments into the following four cases, according to $|npq-m|$ and $|n'pq-m|$.

\begin{itemize}
\item[Case 1:] $|npq-m|=0$ (Lemma \ref{lemma:Case1})
\item[Case 2:] $|npq-m|=|n'pq-m|=1$ (Lemma \ref{lemma:Case2}). 
\item[Case 3:] $|npq-m|=1$, $|n'pq-m|>1$ (Lemma \ref{lemma:Case3}).
\item[Case 4:] $|npq-m|,|n'pq-m|>1$ (Lemma \ref{lemma:Case4}).
\end{itemize}
It is Case 4 where we use additional assumptions (i)--(iv).

Before starting discussions, we review some known results on chirally cosmetic surgery which will be used in the argument.

A knot $K$ is an \emph{L-space knot} if a Dehn surgery on $K$ yields an L-space. For an L-space knot $K$, its Alexander polynomial $\Delta_K(t)$, normalized so that $\Delta_K(t)=\Delta_K(t^{-1})$ and $\Delta_K(1)=1$ hold, is of the form
\[ \Delta_K(t)=(-1)^k + \sum_{j=1}^{k}(-1)^{k-j}(t^{n_j}+t^{-n_j})\] 
for some $0<n_1<n_2 <\cdots <n_k=g(K)$ \cite[Corollary 1.3]{OzSz1}.
From this property, we have the following.
\begin{proposition}
\label{proposition:a2}
If $K$ is an L-space knot which is not unknot, $a_2(K)\neq 0$.
\end{proposition}
\begin{proof}
The coefficient of $z^{2}$ of the Conway polynomial $\nabla_K(z)$ is given by 
\[ a_2(K)=\frac{1}{2}\Delta''_K(t)=\sum_{j=1}^{k}(-1)^{k-j}n_j^{2} \neq 0.\]
\end{proof}

The relevance of L-space knots and (chirally) cosmetic surgery comes from the following result.

\begin{theorem}\cite[Theorem 1.6]{OzSz2}
\label{theorem:HF}
If $S^{3}_K(r)\cong \pm S^{3}_K(r')$ with $rr'>0$, then $K$ is an L-space knot.
\end{theorem}

Then we turn to the proof of Theorem \ref{theorem:main}.

\begin{lemma}
\label{lemma:Case1}
If $|npq-m|=0$ then 
$S^{3}_{K_{p,q}}(n\slash m) \not \cong -S^{3}_{K_{p,q}}(m\slash n')$.
\end{lemma}
\begin{proof}
$S^{3}_{K_{p,q}}(m\slash n)=S^{3}_{K}(p\slash q) \# L(q,p)$ is reducible but   $S^{3}_{K_{p,q}}(m\slash n')$ is irreducible whenever $n'pq-m \neq 0$ \cite{Sc}. Hence they are not homeomorphic.
 \end{proof}

\begin{lemma}
\label{lemma:Case2}
If $|npq-m|=|n'pq-m|=1$ then 
$S^{3}_{K_{p,q}}(n\slash m) \not \cong -S^{3}_{K_{p,q}}(m\slash n')$.
\end{lemma}
\begin{proof}
We may assume that $npq=m+1$ and $n'pq=m-1$ hence $(n-n')pq=2$.
Therefore we have $(p,q)=(\pm 1,2)$ and consequently
$2n=m+1$ and $2n'=m-1$, or, $-2n=m+1$ and $-2n'=m-1$.
We consider the former case $2n=m+1$ and $2n'=m-1$. The latter case is similar.

Since $S^{3}_{K_{p,q}}(m\slash n)= S^{3}_{K}(m \slash 4n) = S^{3}_{K}(m \slash 2m+2)$ and $S^{3}_{K_{p,q}}(m\slash n')= S^{3}_{K}(m \slash 4n') = S^{3}_{K}(m \slash 2m-2)$, we have a chirally cosmetic surgery on the knot $K$
\[ S^{3}_{K}(m \slash 2m+2) \cong -S^{3}_{K}(m \slash 2m-2).\] 
Since $(m \slash 2m+2)(m \slash 2m-2)>0$, i.e., the sign of two surgery slopes are the same, by Theorem \ref{theorem:HF} $K$ is an L-space knot. Hence $a_2(K)\neq 0$ by Proposition \ref{proposition:a2}.

On the other hand, by the surgery formula of Casson-Walker invariant $\lambda$ \cite{Walker}, we have
\begin{align*}
\lambda (S^{3}_{K_{p,q}}(m \slash 2m+2)) &= \frac{2m+2}{m}a_2(K)-\frac{1}{2}s(2m+2,m) \\
\lambda (-S^{3}_{K_{p,q}}(m \slash 2m+2)) &= -\frac{2m-2}{m}a_2(K)+\frac{1}{2}s(2m-2,m).
\end{align*}
here $s(a,b)$ denotes the Dedekind sum. 
Since the Dedekind sum has the properties
\[ s(a,b)=s(a',b) \mbox{ if } a\equiv a' \pmod b, \quad s(-a,b)=-s(a,b), \]
$s(2m+2,m)+s(2m-2,m)=0$.
Since $\lambda (S^{3}_{K_{p,q}}(m \slash 2m+2)) = \lambda (-S^{3}_{K_{p,q}}(m \slash 2m+2))$ we have 
\[ 8a_2(K) =s(2m+2,m)+s(2m-2,m)=0.\]
This is a contradiction.
\end{proof}

\begin{lemma}
\label{lemma:Case3}
If $|npq-m|=1$ and $|n'pq-m|>1$ then 
$S^{3}_{K_{p,q}}(n\slash m) \not \cong -S^{3}_{K_{p,q}}(m\slash n')$.
\end{lemma}
\begin{proof}
Let $k$ be the number of JSJ tori of $E(K)$, and let $X_0$ be the JSJ piece of $E(K)$ that contains $\partial E(K)$.
When $X_0$ is hyperbolic, the simplicial volume of its exterior satisfies $||E(K)||\geq ||E(X_0)||>0$. Since the simplicial volume strictly decreases under Dehn fillings when it is non-zero,
\[ ||S^{3}_{K_{p,q}}(m\slash n)|| = ||S^{3}_{K}(m\slash 4n)|| < || E(K) ||\]
On the other hand,
\[ ||S^{3}_{K_{p,q}}(m\slash n')||= ||E(K) \cup_{T^{2}} P_{p,q,m,n'}|| =||E(K)||\]
we conclude $S^{3}_{K_{p,q}}(m\slash n) \neq -S^{3}_{K_{p,q}}(m \slash n')$.

When $X_0$ is Seifert fibered, $S^{3}_{K_{p,q}}(m\slash n) = S^{3}_{K}(m\slash nq^{2})$ has at most $k$ essential tori whereas $-S^{3}_{K_{p,q}}(m\slash n')$ contains $(k+1)$ essential tori so they are not homeomorphic.
\end{proof}

To treat Case 4, we give a more precise description of the Seifert fibered piece $P_{p,q,m,n}$ and how $E(K)$ is attached to $P_{p,q,m,n}$. 

In the following we use Hatcher's notation $M(g,b;\frac{\alpha_1}{\beta_1},\ldots,\frac{\alpha_n}{\beta_n})$ for Seifert fibered manifold \cite{Hatcher}. For a compact oriented surface with genus $g$ and $b$ boundary components $B$, let $B':= B \setminus (D_1\cup\cdots \cup D_n)$ where $D_1,\ldots,D_n \subset \textrm{Int}(B)$ are disjoint disks. Let $\pi: M' \rightarrow B'$ be the circle bundle over $B'$ with orientable total space. By taking a cross section $\sigma: B' \rightarrow M$ we identify the total space $M'$ with $\sigma(B') \times S^{1}=B'\times S^{1}$. For each torus boundary component $T$ of $M'$ we have a canonical ordered basis given by $([c_T]:= [ B'\times \{*\}\cap T ],[h]:=[\{*\}\times S^{1}\})$ which we call a \emph{section-regular fiber basis}. $M(g,b;\frac{\alpha_1}{\beta_1},\ldots,\frac{\alpha_n}{\beta_n})$ is a 3-manifold obtained by attaching $n$ tori along each torus boundary $T_i:=\partial D_i' \times S^{1}$ so that the slope $\frac{\alpha_i}{\beta_i}$ bounds a disk.

Let $C=C_{p,q}:= (S^{1} \times D^{2}) \setminus N(T_{p,q})$ be the cable space, the complement of the regular neighborhood of the $(p,q)$ torus knot $T_{p,q}$ in a solid torus $S^{1} \times D^{2}$. We fix integers $s,r$ so that $pr+qs=1$. With a suitable choice of section the cable space $C_{p,q}$ is identified with $ M(0,2;\frac{r}{q})$.

Besides a section-regular fiber basis, the boundaries of $C_{p,q}$ has another natural ordered basis. Let $\partial_1 C :=\partial N(T_{p,q})$. By viewing $T_{p,q}$ as usual $(p,q)$ torus knot lying in $S^{1}\times D^{2} \subset S^{3}$, we have the standard meridian-longitude basis $(\mu,\lambda)$ of $H_1(\partial_1 C;\Z)$. 
In terms of the meridian-longitude basis, the section-regular fiber basis $([c_1],[h])$ is written by 
\[ [c_1]=-[\mu], \ [h]=pq[\mu]+[\lambda] \in H_{1}(\partial_1 C;\Z).\]
Since $m[\mu]+n[\lambda]=(npq-m)[c_1]+n[h]$, we have an identification
\[ P_{p,q,m,n}=M(0,1;\frac{r}{q},\frac{n}{npq-m})\] 

For $\partial_2 C := \partial (S^{1} \times D^{2})$ we have a natural basis $([M]=[\{*\}\times \partial D^{1}],[L]=[S^{1}\times\{*\}])$ of $H_1(\partial_2 C;\Z)$ which we call \emph{outer torus basis}.
In terms of the section-regular fiber basis $([c_2],[h])$, the outer torus basis $([M],[L])$ is written by 
\[ [M]=q[c_2]-r[h], \ [L]=p[c_2] +s[h]  \in H_{1}(\partial_2 C;\Z).\]
By definition of cabling, the exterior $E(K)$ is glued to $P_{p,q,m,n}$ by the homeomorphism $\varphi: \partial E(K) \rightarrow \partial P_{p,q,m,n}$ such that $\varphi(\mu_{K})=[M]$ and $\varphi(\lambda_K)=[L]$.

\begin{lemma}
\label{lemma:Case4}
Assume that one of the following conditions are satisfied.
\begin{itemize}
\item[(i)] $q\neq 2$.
\item[(ii)] $q=2$, $p\neq \pm 1$, and the JSJ piece of $E(K)$ does not contain $(-p,2)$-torus knot exterior. 
\item[(iii)] $q=2$, $p=\pm 1$ and the JSJ piece of $E(K)$ does not contain $(r,2)$-torus knot exterior for any $r$.
\item[(iv)] $q=2$, $p=\pm 1$ and $a_2(K)\neq 0$.
\end{itemize} 
Then for $|npq-m|,|n'pq-m|>1$ with $n\neq n'$, 
$S^{3}_{K_{p,q}}(m/n) \not \cong -S^{3}_{K_{p,q}}(m/n')$.
\end{lemma}
\begin{proof}
Assume, to the contrary that $S^{3}_{K_{p,q}}(m/n) \cong -S^{3}_{K_{p,q}}(m/n')$ so there is an orientation preserving homeomorphism $f: S^{3}_{K_{p,q}}(m/n) \rightarrow -S^{3}_{K_{p,q}}(m/n')$. 

By isotopy we assume that $f$ induces homeomorphisms of JSJ pieces.
By the assumption $|npq-m|,|n'pq-m|>1$, $S^{3}_{K_{p,q}}(m/n)$ and $-S^{3}_{K_{p,q}}(m/n')$ have distinguished JSJ piece $P_{p,q,m,n}$ and $-P_{p,q,m,n'}$. Let $X_0=P_{p,q,m,n}$, and $Y_0=f(X_0)$.

\begin{claim}
\label{claim:1}
$Y_0 \neq -P_{p,q,m,n'}$.
\end{claim}

\begin{proof}
Assume to the contrary that $Y_0=-P_{p,q,m,n'}$ so $f$ induces an orientation reversing homeomorphism $f|_P=f|_{P_{p,q,m,n}}: P_{p,q,m,n} \rightarrow P_{p,q,m,n'}$. By uniqueness of Seifert fibration, $f_P$ sends the regular fiber $h$ of $P_{p,q,m,n}$ to the regular fiber $h'$ of $P_{p,q,m,n}$.
Since $f|_P$ is orientation reversing, it inverts the orientation of regular fiber hence we have $f([h])=-[h']$. 

On the other hand, $f_{E(K)}$ induces an orientation reversing homeomorphism $f|_{E(K)}: E(K) \rightarrow E(K)$ hence $K$ is amphichieral.
In particular, we have $f([\mu_K])=-[\mu_K], f([\lambda_K])=[\lambda_K]$.

As we have discussed, in $S^{3}_{K_{p,q}}(m/n)$, the outer torus basis $([M], [L])$ of $P_{p,q,m,n}$ are identified with $[\mu_K]$ and $[\lambda_K]$, respectively. Similarly, in $S^{3}_{K_{p,q}}(m/n')$, the outer torus basis $([M'],[L'])$ of $P_{p,q,m,n'}$ are identified with $[\mu_K]$ and $[\lambda_K]$, respectively. Therefore $f([M])=-[M']$ and $f([L])=[L']$ hence
in terms of the section-regular fiber basis $([c_2],[h])$ and $([c'_2],[h'])$ of $P_{p,q,m,n}$ and $P_{p,q,m,n'}$, we have  
\begin{align*}
f ([M]) & = f (q[c_2]-r[h])=q f([c_2])+r[h'] = -q [c_2']+r[h']=-[M'],\\
f ([L]) & = f (p[c_2]+s[h])=p f([c_2])-s[h'] = p[c'_2]+s[h']=[L'].
\end{align*}
The first equation shows $f([c_2])=-[c'_2]$, which contradicts with the second equation.
\end{proof}

Thus $Y_0=f(X_0)\cong P_{p,q,m,n}$ is a JSJ piece $-E(K)$. Hence there exists a JSJ piece $X_1$ of $E(K)=-(-E(K))$ which is homeomoprhic to $-Y_0 \cong -P_{p,q,m,n}$.

The next claim, together with our assumption (ii), shows that such a JSJ piece cannot be send to $-P_{p,q,m,n'}$.

\begin{claim}
\label{claim:A}
Let $X$ be a JSJ piece of $E(K)$ which is homeomorphic to $-P_{p,q,m,n}$.
If $f(X) =-P_{p,q,m,n'}$, then $q=2$ and $X$ is homeomoprhic to $(-p,2)$-torus knot exterior. Moreover, $K_{p,q}$ is an L-space knot.
\end{claim}

\begin{proof}[Proof of Claim \ref{claim:A}]
Since $-X\cong P_{p,q,m,n}$ is a Seifert fibered space with disk base and two singular fibers that appears as a JSJ piece of the knot exterior $E(K)$, $X$ is homeomorphic to the torus knot exterior $E(T_{P,Q})$ for some $P,Q$.
We fix integers $S,R$ so that $PR+QS=1$. If $f(X)=-P_{p,q,m,n'}$,
\begin{align*}
-X \cong E(T_{P,Q}) \cong M(0,1;\frac{R}{Q}, \frac{S}{P})  & \cong M(0,1;\frac{r}{q}, \frac{n}{npq-m}) \cong P_{p,q,m,n} \\
 &\cong M(0,1; \frac{r}{q}, \frac{n'}{n'pq-m}) \cong P_{p,q,m,n'} 
\end{align*} 
Thus we may assume that we have $q=Q$, $r=R$, $P=npq-m=-(n'pq-m)$ and that there are integers $i,j$ such that 
\[ \frac{n}{npq-m}+i =\frac{S}{P}, \qquad \frac{-n'}{npq-m}+j =\frac{S}{P}.\] 
In particular we have
\begin{equation}
\label{eqn:A1}
\begin{cases}
r(npq-m)+qn+qi(npq-m)=1, \\
r(npq-m)-qn'+qj(npq-m)=1. \\
\end{cases}
\end{equation}
Since $P=npq-m=-(n'pq-m)$, we have $(n+n')pq=2m$. By (\ref{eqn:A1}), $q$ and $m$ are coprime hence we have $q=2$. Consequently, we get $r=R=1$, $q=Q=2$, and 
\[ P= npq-m = 2np - (n+n')p= (n-n')p. \]
Then (\ref{eqn:A1}) is written by
\[
\begin{cases}
(n-n')p+2n+2i(n-n')p=1, \\
(n-n')p-2n'+2j(n-n')p=1. \\
\end{cases}\]
So we have $(n-n')(1+p+ip+jp)=1$ hence $(n-n')=\pm 1$. 
If $(n-n')=1$, we have $P=(n-n')p=p$ so $X \cong -P_{p,q,m,n} \cong -E(T_{P,2}) \cong E(T_{-p,2})$. Moreover, $n-n'=1$ means that the signs of surgery slopes $m/n$ and $m/n'$ are the same hence $K_{p,q}$ is an L-space knot by Theorem \ref{theorem:HF}. 

If $(n-n')=-1$, we have $p(1+i+j)=-2$ so $p=\pm 1$. Then we have $|npq-m|=|(n-n')p|=1$ so it contradicts the assumption.

\end{proof}
 
Thus $Y_1=f(X_1)$ is a JSJ piece of $-E(K)$. Hence we have a JSJ piece $X_2$ of $E(K)$ which is homeomoprhic to $P_{p,q,m,n}$.
The next claim, similar to Claim \ref{claim:A}, together with the assumptions (iii) and (iv) shows that such a JSJ piece cannot be send to $-P_{p,q,m,n'}$, either.

\begin{claim}
\label{claim:B}
Let $X$ be a JSJ piece of $E(K)$ which is homeomorphic to $P_{p,q,m,n}$. If $f(X)= -P_{p,q,m,n'}$ then $q=2$, $p = \pm1 $ and $a_{2}(K)=0$. Moreover $X$ is homeomoprhic to $(2,\pm(n-n'))$-torus knot exterior.
\end{claim}
\begin{proof}[Proof of Claim \ref{claim:B}]
As in Claim \ref{claim:A}, $X$ is homeomorphic to the torus knot exterior $E(T_{P,Q})$ for some coprime $P,Q$ and we have 
\begin{align*}
X \cong E(T_{P,Q}) \cong M(0,1;\frac{R}{Q}, \frac{S}{P}) & \cong M(0,1;\frac{r}{q}, \frac{n}{npq-m}) \cong P_{p,q,m,n} \\
 &\cong M(0,1; -\frac{r}{q}, -\frac{n'}{n'pq-m}) \cong -P_{p,q,m,n'}.
\end{align*} 
Here $S,R$ are integers chosen so that $PR+QS=1$. We may assume that $q=Q$, $r=R$, $P=npq-m$.

We have either $|npq-m|=|n'pq-m|$ or $|npq-m|=|q|$. In the latter case we also have $|n'pq-m|=|q|=|n'pq-m|$ so in both cases we always have $|npq-m|=|n'pq-m|$. Since $n\neq n'$ we have $npq-m=-n'pq+m$ so $(n+n')pq=2m$.

On the other hand, there is an integer $i$ such that $\frac{n}{npq-m}+i =\frac{S}{P}$ so we have $r(npq-m)+qn+qi(npq-m)=1$. This implies that $m$ and $q$ are coprime so we have $q=Q=2$. Consequently, $(n+n')p=m$ hence $npq-m=2np-(n+n')p=(n-n')p$.

By comparing Seifert invariants, we have integers $i,j$ such that 
\[
\begin{cases}
(n-n')p+2n+2i(n-n')p=1, \\
(n-n')p+2n'+2j(n-n')p=1 \\
\end{cases}
\]
so we have $(n-n')(ip-jp+1)=0$. 
Consequently, we have $(i-j)p=-1$ so $p=\pm 1$. Thus $P=npq-m=(n-n')p=\pm (n-n')$.

Also, by $p=\pm 1$ we have $n+n' = \pm m$. This shows that $n+n' \equiv 0 \pmod m$. By the Casson-Walker invariant we have
\begin{align*}
 \lambda(S^{3}_{K_{\pm 1,2}}(m\slash n)) & = \frac{n}{m}a_2(K_{\pm 1,2})-\frac{1}{2}s(n,m)\\
& = -\frac{n'}{m}a_2(K_{\pm 1,2})+\frac{1}{2}s(n',m) = \lambda(-S^{3}_{K_{\pm 1,2}}(m\slash n'))
\end{align*}
we have
\[ \frac{n+n'}{m}a_2(K_{\pm 1,2}) = \frac{1}{2}(s(n,m)+s(n',m))=0\]
Since $n+n' \neq 0$ because this implies $m=0$, we have $a_{2}(K_{p,q})=0$.
On the other hand, since $\Delta_{K_{p,q}}(t)=\Delta_{K}(t^{q})\Delta_{T_{p,q}}(t)$ we have $a_{2}(K_{p,q})=q^{2}a_2(K)+a_2(T_{p,q})$.
Thus $a_2(K_{\pm 1,2})=4a_{2}(K)=0$.

\end{proof}

Therefore $Y_2=f(X_2)$ appears as a JSJ piece of $-E(K)$ hence we have a JSJ piece $X_3$ of $E(K)$ which is homeomoprhic to $-P_{p,q,m,n}$.

Then we repeat the argument; for each $i>2$, we have a JSJ piece $X_i$ which is homeomorphic to $-P_{p,q,m,n}$ (if $i$ is odd) or $P_{p,q,m,n}$ (if $i$ is even).
Then by assumption of lemma and Claim \ref{claim:A} (if $i$ is odd) or Claim \ref{claim:B} (if $i$ is even), we see that $f(X_i) \neq -P_{p,q,m,n'}$ hence 
 $Y_i=f(X_i)$ gives a new JSJ piece of $-E(K)$.
This means that we find a new JSJ piece $X_{i+1}$ in $E(K)$, homeomorphic to $-P_{p,q,m,n}$ (if $i$ is even) or $P_{p,q,m,n}$ (if $i$ is odd) (see Figure \ref{fig:chart} below for a schematic illustration). Thus $E(K)$ contains infinitely many JSJ piece, which is absurd. 

\begin{figure}[htbp]
\includegraphics*[width=120mm,bb= 124 541 506 721]{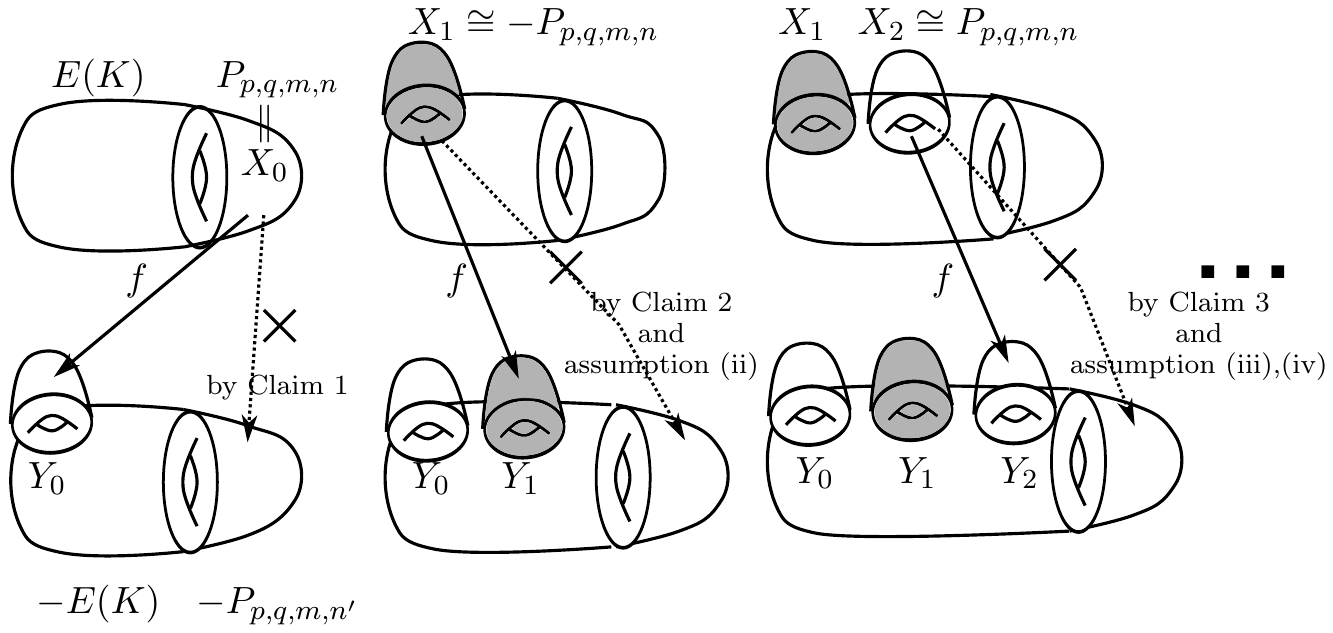}
\caption{Proof of Lemma 4: $S^{3}_{K_{p,q}}(n\slash m) \not \cong -S^{3}_{K_{p,q}}(m\slash n')$ imposes infinitely many JSJ pieces.} 
\label{fig:chart}
\end{figure} 

\end{proof}

\section{Iterated cables}

For a sequence $(p_1,q_1),\ldots,(p_N,q_N)$ of coprime integers with $q_i>1$ and a knot $K$, we define an iterated cable $K_{(p_1,q_1),\ldots,(p_N,q_N)}$ inductively by
\[
K_{(p_1,q_1)}=K_{p_1,q_1}, \quad
K_{((p_1,q_1),\ldots,(p_N,q_N))}= (K_{((p_1,q_1),\ldots,(p_{N-1},q_{N-1}))})_{(p_N,q_N)}.
\]
When $K$ is the unknot $U$, the iterated cable $U_{(p_1,q_1),\ldots,(p_N,q_N)}$ is called the \emph{iterated torus knot}.

We prove a theorem which is slightly general than Theorem \ref{theorem:main2}, by adding more arguments to Lemma \ref{lemma:Case4}.

\begin{theorem}
\label{theorem:main2-general}
Let $K$ be a non-satellite knot. Then an iterated cable $K_{(p_1,q_1),\ldots,(p_N,q_N)}$ for $N\geq 1$ does not admit chirally cosmetic surgery.
\end{theorem}

\begin{proof}[Proof of Theorem \ref{theorem:main2-general}]
An iterated cable of torus knot is an iterated torus knot so we may assume that $K$ is either hyperbolic or unknot.
We put $p=p_N$, $q=q_N$ and view the iterated cable $K_{(p_1,q_1),\ldots,(p_N,q_N)}$ as $K^{*}_{(p,q)}$, the $(p,q)$-cable of the iterated cable $K^{*}=K_{(p_1,q_1),\ldots,(p_{N-1},q_{N-1})}$.

The JSJ decomposition of $E(K^{*})$ is given by 
\[ E(K^{*})= E(T_{p_1,q_1}) \cup_{T_1} C_{p_2,q_2} \cup_{T_2} \cdots \cup_{T_{N-2}}C_{p_{N-1},q_{N-1}} \]
if $K$ is unknot, and 
\[ E(K^{*}) = E(K) \cup_{T_0}  C_{p_1,q_1}\cup_{T_1} C_{p_2,q_2} \cup_{T_2} \cdots \cup_{T_{N-2}}C_{p_{N-1},q_{N-1}} \]
otherwise (i.e., $K$ is hyperbolic).
When $K$ is hyperbolic, no JSJ piece of $E_{K^{*}}$ is homeomorphic to the torus knot exterior so by Theorem \ref{theorem:main}, $K^{*}_{p,q}$ does not admit chirally cosmetic surgery. Thus we assume that $K^{*}$ is an iterated torus knot. Since the classification of chirally cosmetic surgery of torus knots are known \cite{IIS,Rong}, in the following we assume that $K$ is not a torus knot.

Assume, to the contrary that $S^{3}_{K^{*}_{p,q}}(m/n) \cong -S^{3}_{K^{*}_{p,q}}(m/n')$ so there is an orientation preserving homeomorphism $f: S^{3}_{K^{*}_{p,q}}(m/n) \rightarrow -S^{3}_{K^{*}_{p,q}}(m/n')$. 
By Lemma \ref{lemma:Case1}, \ref{lemma:Case2}, \ref{lemma:Case3}, $|n pq -m|, |n'pq -m|>1$ hence $S^{3}_{K^{*}_{p,q}}(m/n) = E(K^{*}) \cup_T P_{p,q,m,n}, S^{3}_{K^{*}_{p,q}}(m/n') = E(K^{*}) \cup_T P_{p,q,m,n'}$.

By isotopy we assume that $f$ induces homeomorphisms of JSJ pieces.
By Claim \ref{claim:1} in Lemma \ref{lemma:Case4}, $f(P_{p,q,m,n})$ is a JSJ piece of $-E(K^{*})$. 
Since the cable space $C_{p_i,q_i}$ has two boundary components whereas the boundary of $P_{p,q,m,n}$ is connected, we have $f(P_{p,q,m,n}) = -E(T_{p_1,q_1})$. 
Since $f(E(T_{p_1,q_1}))$ is a JSJ piece of $-S^{3}_{K^{*}_{p,q}}(m/n') = -E(K^{*}) \cup_T -P_{p,q,m,n'} $ other than $-E(T_{p_1,q_1})$, we have $f(E(T_{p_1,q_1})) = -P_{p,q,m,n'}$.
This shows that $f$ gives an orientation homeomorphism
\[ f: -f(P_{p,q,m,n}) = E(T_{p_1,q_1}) \rightarrow -P_{p,q,m,n'}.\]

By Claim \ref{claim:A} in Lemma \ref{lemma:Case4}, we have $q_1=2,p_1=-p$ and $K=K^{*}_{p,q}$ is an L-space knot. On the other hand, by \cite{Hom} an iterated torus knot $K_{(-p,2),\ldots,(p_{N-1},q_{N-1}),(p,2)}$ is an L-space knot implies that $-p,p_2,p_3,\ldots,p$ has the same sign. This is a contradiction. 
\end{proof}

\end{document}